\documentclass[runningheads]{llncs}
\usepackage{amsmath, amssymb}
\usepackage{algorithmic}

\usepackage{algorithm}
\usepackage{amsfonts}

\usepackage{graphicx}
\graphicspath{{fig/}}

\usepackage{url}
\usepackage[unicode, pdftex]{hyperref}

\usepackage{amsmath}
\usepackage{amsfonts}
\usepackage{amssymb}
\usepackage{mathrsfs}
\usepackage{stmaryrd}

\newcommand{\Gr}{Gr\"obner }

\newcommand{\cI}{\cal{I}}
\newcommand{\cR}{\cal{R}}

\newcommand{\Q}{\mathbb{Q}}
\newcommand{\Z}{\mathbb{Z}}
\newcommand{\K}{\cal{K}}
\newcommand{\N}{\mathbb{N}}
\newcommand{\M}{\cal{M}}
\newcommand{\R}{\mathbb{R}}

\newcommand{\lm}{\mathop{\mathrm{lm}}\nolimits}

\newcommand{\lt}{\mathop{\mathrm{lt}}\nolimits}
\newcommand{\lc}{\mathop{\mathrm{lc}}\nolimits}

\urldef{\mailsa}\path|gerdt@.jinr.ru|

\begin{document}

\mainmatter              

\title{On Consistency of Finite Difference Approximations to the Navier-Stokes Equations}
\titlerunning{Consistency of Approximations to Navier--Stokes Equations}
\author{Pierluigi Amodio\inst{1},
Yuri Blinkov\inst{2},
Vladimir Gerdt\inst{3}  and
Roberto La Scala\inst{1}
}
\authorrunning{P. Amodio, Yu. Blinkov, V. Gerdt, and R. La Scala}

\institute{Department of Mathematics, University of Bari, Bari, Italy\\
 \email{\{pierluigi.amodio, roberto.lascala\}@uniba.it}
\and
Department of Mathematics and Mechanics, Saratov State University, Saratov, Russia\\
\email{BlinovUA@info.sgu.ru}
\and
Laboratory of Information Technologies, Joint Institute for Nuclear Research, Dubna, Russia\\
\email{gerdt@jinr.ru}
}
\vskip -1.0cm

\maketitle              
\tocauthor{
Pirluigi Amodio (University of Bari, Bari)
Yuri Blinkov (Saratov State University, Saratov)
Vladimir Gerdt (Joint Institute for Nuclear Research, Dubna)
Roberto La Scala (University of Bari, Bari)}

\begin{abstract}
In the given paper, we confront three finite difference
approximations to the Navier--Stokes equations for the
two-dimensional viscous incomressible fluid flows. Two of these
approximations were generated by the computer algebra assisted
method proposed based on the finite volume method, numerical
integration, and difference elimination. The third approximation
was derived by the standard replacement of the temporal
derivatives with the forward differences and the spatial
derivatives with the central differences. We prove that only one
of these approximations is strongly  consistent with the
Navier--Stokes equations and present our numerical tests which
show that this approximation has a better behavior than the other
two.
\end{abstract}




\section{Introduction}

By its completion to involution~\cite{Seiler'10}, the well-known
Navier--Stokes system of equations~\cite{Pozrikidis:2001} for
unsteady two-dimensional motion of incompressible viscous liquid
of constant viscosity may be written in the following
dimensionless form~\cite{GB'09}
\begin{eqnarray}
\left\lbrace
\begin{array}{l}
f_1 := u_x+v_y=0\,,\\[0.1cm]
f_2 := u_t + u u_x + v u_y + p_x - \frac{1}{\mathrm{Re}} (u_{xx}+u_{yy})=0\,,\\[0.1cm]
f_3 := v_t + u v_x + v v_y + p_y - \frac{1}{\mathrm{Re}} (v_{xx}+v_{yy})=0\,,\\[0.1cm]
f_4 := u_x^2 + 2 v_x u_y + v_y^2 + p_{xx} + p_{yy} = 0 \,.
\end{array}
\right.\label{ins}
\end{eqnarray}
Here $(u,v)$ is the velocity field, $f_1$ is the continuity
equation, $f_2$ and $f_3$ are the proper Navier--Stokes
equations~\cite{Pozrikidis:2001}, and $f_4$ is the pressure
Poisson equation~\cite{Gresho:1987}. The constant $\mathrm{Re}$
denotes the Reynolds number.

For discretization we use the finite difference
method~\cite{Samarskii'01,Str'04} and consider orthogonal and
uniform computational grid. In this method, the finite difference
approximation (FDA) to the differential equations combined with
appropriate initial or/and boundary conditions in their discrete
form constitutes the finite difference scheme (FDS) for
construction of a numerical solution. The main requirement to the
scheme is convergence of its numerical solution to the solution of
differential equation(s) when the grid spacings go to zero.

The fundamental problem in numerical solving of partial
differential equation (PDE) or a system of PDEs is to construct
such FDA that for any initial- or/and boundary-value problem,
providing existence and uniqueness of the solution to PDE(s) with
a smooth dependence on the initial or/and boundary data, the
corresponding FDS is convergent. For polynomially-nonlinear PDEs,
e.g., the Navier--Stokes equations, to satisfy this requirement
FDA must inherit all algebraic properties of the differential
equation(s). The necessary condition for the inheritance is the
property of $s$(strong)-consistency of FDA to PDEs introduced first
in~\cite{GR'10} for linear equations and extended in~\cite{G'12}
to nonlinear ones.

The conventional consistency~\cite{Str'04}, called
in~\cite{GR'10,G'12} by weak consistency ($w-$ consistency) implies
reduction of FDA to the original PDE(s) when the grid spacings go
to zero. This consistency can be verified by a Taylor expansion of the
difference equations in the FDA about a grid point. The strong
consistency ($s-$consistency) implies reduction of any element in the perfect
difference ideal generated by the FDA to an element in the radical
differential ideal generated by the PDE(s). In \cite{G'12}, it was
shown that $s-$consistency can be checked in terms of a difference
\Gr basis of the ideal generated by the FDA. Since difference
polynomial ring~\cite{Levin'08} is non Noetherian,  in the
nonlinear case, generally, one cannot verify $s-$consistency of a
given FDA through computation of associated difference \Gr basis.
However, if the FDA under consideration is $w-$consistent, then it
is not $s-$consistent if and only if at some  step of the
Buchberger-like algorithm (cf.~\cite{L'12,GLS'13} and~\cite{G'12})
applied to construction of the \Gr basis,  a  difference
S-polynomial arises which in not $w-$consistent with any of the
consequences of the original PDE(s). In practice, this may help to
detect $s-$inconsistency.

In \cite{GB'09}, the algorithmic approach to generation of FDA
suggested in~\cite{GBM'06} was applied to the Navier--Stokes
equations (\ref{ins}). The approach is based on the finite volume
method combined with numerical integration and difference
elimination. As a result, three different $w-$consistent FDAs were
obtained in \cite{GB'09}. Two of them were analyzed in~\cite{G'12}
from the viewpoint of $s-$consistency. One of these FDAs was
qualified as a "good" one, i.e., $s-$consistent, by the claim that
it itself is a \Gr basis. Another FDA was qualified as
s-inconsistent by inspection (observed already in~\cite{GB'09})
that one of its differential consequences is not reduced to a
differential consequence of the system (\ref{ins}) when the grid
spacings go to zero. However, as explicit computation with the
Maple-based implementation~\cite{L'12} of the Buchberger-like
algorithm~\cite{L'12,GLS'13} showed, the ``good'' FDA is not a \Gr
basis what generates a need for the further investigation of its
s-consistency.

In this paper, we prove that the "good" FDA generated
in~\cite{GB'09} is indeed $s-$consistent. In doing so we avoid the
\Gr basis computation what is rather cumbersome. In addition, we
consider universally adopted standard method to discretization,
which consists in the replacement of the temporal derivatives in
(\ref{ins}) with the forward differences and the spatial
derivatives with the central differences, and show that it yields
FDA which is not $s-$consistent. To see numerical impact of the
property of s-consistency we confronted the three FDAs and
compared their behavior for the mixed initial-boundary value
problem whose data originate from the exact solution~\cite{Chorin}
of (\ref{ins}). This comparison clearly shows superiority of the
$s-$consistent FDA over the other two which are not $s-$consistent.

This paper is organized as follows. Section 2 introduces the main
objects of difference algebra related to discretization of
(\ref{ins}). Section 3 is concerned with definition of FDA to
(\ref{ins}) and its s-consistency. In Section 4, we consider three
particular FDAs to the Navier--Stokes system (\ref{ins}) and
establish their s-consistency properties. Section 5 presents the
results of our numerical computer experiments. Some concluding
remarks are given in Section 6.

\section{Preliminaries}

The left-hand sides of the PDEs in the Navier--Stokes system
(\ref{ins}) can be considered as elements in the {\em differential
polynomial ring}~\cite{Hubert'01}
\begin{equation}
f_i=0\ (1\leq i\leq 4),\quad F:=\{f_1,f_2,f_3,f_4\}\subset R:={\K}[u,v,p]\,,  \label{pde}
\end{equation}
where $\{u,v,p\}$ is the set of {\em differential indeterminates} and
${\K}:={\Q}({\mathrm{Re}})$ is the {\em differential field of constants}.

\begin{remark}
It is easy to check that the differential ideal $[F]\subset R$ generated by $F$ is {\em radical}~\cite{Hubert'01}.
\end{remark}

To approximate the differential system (\ref{pde}) by a difference
one, we use an orthogonal and uniform computational grid (mesh) as
the set of points $(jh,kh,n\tau)\in \R^3$.  Here $\tau>0$ and
$h>0$ are the grid spacings (mesh steps), and the triple of
integers $(j,k,n)\in \Z^3$ numerates the grid points. In doing so,
in a grid node $(jh,kh,n\tau)$ a solution to (\ref{ins}) is
approximated by the triple of grid functions
\begin{equation}
\{u^n_{j,k},v^n_{j,k},p^n_{j,k}\}:=\{u,v,p\}\mid_{x=jh,y=kh,t=\tau n}.\label{grid_functions}
\end{equation}

Now we introduce the set of mutually commuting {\em differences}
$\{\sigma_x,\sigma_y,\sigma_t\}$ acting on a grid function
$\phi(x,y,t)$, which is to approximate a solution of (\ref{ins})
on the grid points, as the forward shift operators
\begin{eqnarray}
\left\lbrace
\begin{array}{l}
\sigma_x\circ \phi(x,y,t)=\phi(x+h,y,t)\,,\\[0.1cm]
\sigma_y\circ \phi(x,y,t)=\phi(x,y+h,t)\,,\\[0.1cm]
\sigma_t\circ \phi(x,y,t)=\phi(x,y,t+\tau)\,.
\end{array}
\right.\label{shift}
\end{eqnarray}

The {\em monoid} generated by the differences will be denoted by
$\Sigma$, i.e.,
\[
\Sigma:=\{\,\sigma_x^{i_1}\sigma_y^{i_2}\sigma_t^{i_3}\mid i_1,i_2,i_3\in \N_{\geq 0}\,\}\,,\quad
(\,\forall \sigma\in \Sigma\,)\ [\,\sigma \circ 1=1\,]\,,
\]
and the ring of {\em difference polynomials} over ${\K}$ will be
denoted by ${\cR}$. The elements in ${\cR}$ are polynomials in the
{\em difference indeterminates} $u,v,p$ (dependent variables)
defined on the grid points and in their shifted values
\[
\{\ \sigma_x^{i_1}\sigma_y^{i_2}\sigma_t^{i_3}\circ w\mid w\in \{u,v,p\},\ \{i_1,i_2,i_3\}\in \N_{\geq 0}^3\ \}\,.
\]

\begin{definition}{\em\cite{Levin'08}}
A total order $\prec$ on $\{\,\sigma\circ w\mid \sigma \in
\Sigma,\ w\in \{u,v,p\}\,\}$ is {\em ranking} if for all
$\sigma,\sigma_1,\sigma_2,\sigma_3\in \Sigma$ and $w,w_1,w_2\in
\{u,v,p\}$
\[
\sigma\,\sigma_1\,\circ w \succ \sigma_1\circ w\,, \quad
\sigma_1 \circ w_1 \succ \sigma_2\circ w_2\,
\Longleftrightarrow \, \sigma\circ \sigma_1\circ w_1 \succ \sigma\circ \sigma_2\circ w_2
\]
\label{def-diffc-ranking}
\end{definition}
The set ${\M}$ of {\em monomials} in the ring ${\cR}$ reads
\begin{equation}
{\M}:=\{\,(\sigma_1\circ u)^{i_1}(\sigma_2\circ v)^{i_2}(\sigma_3\circ p)^{i_3}\mid \sigma_j\in \Sigma,\ i_j\in \N_{\geq 0},\ 1\leq j\leq 3\,\}\,. \label{monomials}
\end{equation}

\begin{definition}{\em\cite{G'12}}
A total order $\succ$ on ${\M}$ is {\em admissible} if it extends a ranking and
\[
(\forall\, \mu\in {\M}\setminus \{1\})\, [\mu\succ 1]\, \wedge\,
(\,\forall\, \sigma\in \Sigma)\, (\,\forall\, \mu,a,b\in {\M}\,)\, [\,a\succ b\Longleftrightarrow \mu\cdot\sigma\circ a\succ  \mu\cdot \sigma\circ b\,]\,.
\]
\label{mon-order}
\end{definition}

Given an admissible monomial order $\succ$, every difference
polynomial $p$ has the {\em leading monomial} $\lm(p)\in {\M}$ and
the {\em leading term} $\lt(p):=\lm(p)\lc(p)$  with the {\em
leading coefficient} $\lc(p)$. Throughout this session, every
difference polynomial is to be {\em normalized (i.e., monic)} by
division of the polynomial by its leading coefficient. This
provides  $(\,\forall p\in {\cR}\,)\ [\,\lc(p)=1\,]$.

Now we consider the notions of difference ideal~\cite{Levin'08}
and its standard basis. The last notion was introduced
in~\cite{G'12} in the full analogy to that in differential
algebra~\cite{Ollivier'90}.

\begin{definition}{\em\cite{Levin'08}} A set ${\cI}\subset {\cR}$ is {\em difference polynomial ideal}  or {\em $\sigma$-ideal} if
\[
(\,\forall\, a,b\in {\cI}\,)\ (\,\forall\, c\in {\cR}\,)\  (\,\forall\, \sigma \in \Sigma\,)\ [\,
a+b\in {\cI},\ a\cdot c\in {\cI},\ \sigma\circ a\in {\cI}\,].
\]
If $F\subset {\cR}$, then the smallest $\sigma$-ideal containing $F$ is said to be generated by $F$ and denoted by $[F]$.
\label{diffc-ideal}
\end{definition}

\begin{definition}{\em\cite{G'12}}
If for $\alpha,\beta\in {\M}$ the equality $\beta=\mu\cdot
\sigma\circ \alpha$ holds  with $\sigma\in \Sigma$ and $\mu\in
{\M}$ we shall say that $\alpha$ {\em divides} $\beta$ and write
$\alpha\mid \beta$. It is easy to see that this divisibility
relation yields a {\em partial order}. \label{division}
\end{definition}

\begin{definition}{\em\cite{G'12}} Given a $\sigma$-ideal ${\cI}$ and an admissible
monomial ordering $\succ$, a subset $G\subset {\cI}$ is its {\em (difference) standard basis} if $[G]={\cI}$ and
\[
(\,\forall\, p\in {\cI}\,) (\,\exists\, g\in G\,)\ \ [\,\lm(g)\mid \lm(p)\,]\,.
\]
If the standard basis is finite we shall call it {\em \Gr basis}.
\label{def_SB}
\end{definition}

\begin{remark}
Based on Definition~\ref{division}, one can introduce (see~\cite{G'12}) in difference
algebra the concepts of {\em polynomial reduction} and {\em normal form} of a difference
polynomial $p$ modulo a set of difference polynomials $P$ (notation: $\mathrm{NF}(p,P)$).
A {\em reduced standard basis} $G$ is such that $(\,\forall g\in G)\ [\,g=NF(g,G\setminus \{g\})\,].$
\end{remark}
The algorithmic characterization of standard bases and their construction in difference
polynomial rings is done in terms of difference $S$-polynomials.

\begin{definition}{\em\cite{G'12}}
Given an admissible order, and monic difference polynomials $p$
and $q$ (they not need to be  distinct), the polynomial
$S(p,q):=m_1\cdot \sigma_1\circ p-m_2\cdot \sigma_2\circ q$ is
called (difference){\em $S$-polynomial} associated to $p$ and $q$
if $ m_1\cdot \sigma_1\circ  \lm(p)=m_2\cdot \sigma_2\circ \lm(q)$
with co-prime $m_1\cdot \sigma_1$ and $m_2\cdot \sigma_2$.
\label{S-polynomial}
\end{definition}

\begin{remark}
This characterization immediately implies~\cite{G'12,GLS'13} a
difference version of the Buchberger algorithm
(cf.~\cite{CLO'07,Ollivier'90}). The algorithm always terminates
when the input polynomials are linear. If this is not the case,
the algorithm may not terminate. Additionally, one can take into
account Buchberger's criteria to avoid some useless zero
reductions. The difference  criteria are similar to the
differential ones~\cite{Ollivier'90}.
\end{remark}

\begin{definition}{\em\cite{Levin'08}} A {\em perfect difference ideal} generated
by a set $F\subset {\cal{R}}$ and denoted by $\llbracket F\rrbracket$ is
the smallest difference ideal containing $F$ and such that for any
$f\in {\cal{R}}$, $\sigma_1,\ldots,\sigma_r\in \Sigma$ and $k_1,\ldots,k_r \in \N_{\geq 0}$
\[
(\sigma_1\circ f)^{k_1}\cdots (\sigma_r\circ f)^{k_r}\in \llbracket
F\rrbracket \Longrightarrow f\in \llbracket F\rrbracket \,.
\]
\end{definition}

\begin{remark}
In difference algebra, perfect ideals play the same role
(cf.~\cite{Trushin'11}) as radical ideals in commutative and
differential algebra. Obviously, $[F]\subseteq \llbracket
F\rrbracket$.
\end{remark}

\section{Consistency of Difference Approximations}

Let a finite set of difference polynomials
\begin{equation}
\tilde{f}_1=\cdots=\tilde{f}_p=0\,,\quad \tilde{F}:=\{\tilde{f}_1,\ldots \tilde{f}_p\}\subset {\cR} \label{fda}
\end{equation}
be a FDA to (\ref{ins}). It should be noted that generally the
number $p$ in (\ref{fda}) needs not to be equal to the number of
equations in (\ref{ins}).

\begin{definition}
A differential (resp. difference) polynomial $f\in {R}$ (resp.
$\tilde{f}\in {\cR}$) is {\em differential-algebraic} (resp. {\em
difference-algebraic}) {\em consequence} of (\ref{ins}) (resp.
(\ref{fda})) if $f\in \llbracket F\rrbracket $ (resp.
$\tilde{f}\in \llbracket \tilde{F}\rrbracket $).
\end{definition}

\begin{definition}\label{implication} We shall say that a {\em difference equation
$\tilde{f}=0$ implies (in the continuous limit) the differential equation $f=0$} and write
$\tilde{f}\rhd f$ if $f$ does not contain the grid spacings $h,\tau$ and the Taylor expansion
about a grid point $(u^n_{j,k},v^n_{j,k},p^n_{j,k})$ transforms equation $\tilde{f}=0$ into
$f + O(h,\tau)=0$ where $O(h,\tau)$ denotes expression which vanishes when $h$ and $\tau$ go to zero.
\end{definition}

\begin{definition}{\em\cite{G'12,GR'10}}
The difference approximation (\ref{fda}) to (\ref{ins}) is {\em weakly consistent} or
{\em $w-$consistent} with (\ref{ins}) if $p=4$ and
\begin{equation}
   (\,\forall \tilde{f}\in \tilde{F}\,)\ (\,\exists f\in F\,)\ [\,\tilde{f}\rhd f\,]\,. \label{w-cons}
\end{equation}
\label{def-econs}
\end{definition}
The requirement of weak consistency which has been universally accepted in the literature,
is not satisfactory by the following two reasons:
\begin{enumerate}
 \item The cardinality of FDA in (\ref{FDA}) may be different from that of the original set
 of differential equations. For example, the systems $ \{\,u_{xz} + yu = 0,\ u_{yw} + z u = 0\,\}$ and
     $\{\,yu_y-zu_z=0,\ u_x - u_w=0,\ u_{xw} +yu_y=0\,\}$ in one dependent and four dependent variables
     are fully equivalent (see~\cite{GR'10}, Example 3). Thus, to construct a FDA, one can use them
     interchangeably. Whereas Definition \ref{def-econs} fastens $\tilde{F}$ to $F$.
 \item A $w-$consistent FDA may not be good in view of inheritance of properties of differential
 systems at the discrete level. We shall demonstrate this in the next section.
 \end{enumerate}

Another concept of consistency was introduced in~\cite{GR'10} for
linear FDA and then extended in~\cite{G'12} to the nonlinear case.
For the Navier--Stokes system, it is specialized as follows.

\begin{definition}
An FDA (\ref{fda}) to (\ref{ins}) is {\em strongly consistent} or {\em $s-$consistent} if
\begin{equation}
(\,\forall \tilde{f}\in \llbracket \tilde{F} \rrbracket\,)\  (\,\exists
f\in [F]\,)\ [\,\tilde{f}\rhd f\,]\,. \label{s-cond}
\end{equation}
\label{def-scon}
\end{definition}

The algorithmic approach of paper~\cite{G'12} to verification of s-consistency is based on the following theorem.

\begin{theorem}{\em\cite{G'12}}
A difference approximation (\ref{fda}) to (\ref{ins}) is $s-$consistent if and only
if a (reduced) standard basis $G$ of the difference
ideal $[\tilde{F}]$ satisfies
\begin{equation}
(\,\forall g\in G\,)\ (\,\exists f\in [F]\,)\  [\,g\rhd f\,]\,. \label{cons-gb}
\end{equation}
\label{th:s-cons}
\end{theorem}

Irrespective of possible infiniteness of the (nonlinear)
difference standard basis $G$, it may be useful to apply an
algorithm for its construction (see, for example, the algorithms
in~\cite{G'12,GLS'13}) and to verify s-consistency of the
intermediate polynomials. In doing so, one should check first the
$w-$consistency of the polynomials in the input FDA. Then, if the
normal form $\tilde{p}$ of an $S-$polynomial modulo the current
basis is nonzero, then before insertion of $\tilde{p}$ into the
intermediate basis one has to construct $p$ such that
$\tilde{p}\,\rhd p$ and check the condition $p\in \llbracket
F\rrbracket$.

\begin{remark}
Given a differential polynomial $f\in R$, one can algorithmically check
its membership in $\llbracket F\rrbracket$ by performing the involutive Janet reduction~\cite{G'12}.
\end{remark}

\section{Three Difference Approximations to the Navier--Stokes Equations}
To analyze strong consistency of difference approximations to (\ref{ins}) we shall need the following statements.

\begin{proposition}\label{proposition}
Let $\tilde{f}\in R$ be a difference polynomial. Suppose
$\tilde{f}\rhd f$ where $f$ is a differential-algebraic
consequence of the Navier--Stokes system (\ref{ins}), $f\in [F]$.
Then a finite sum of the form
\begin{equation}
 \tilde{p}:=\sum_i \tilde{g}_i\cdot \sigma_i\circ \tilde{f},\quad \sigma_i\in \Sigma,\quad \tilde{g}_i\in {\cR} \label{d_pol}
\end{equation}
also implies a differential-algebraic consequence of (\ref{ins}).
\end{proposition}

\begin{proof}
The shift operators in (\ref{shift}) are expanded in the Taylor series as follows
\[
 \sigma_{x}=\sum_{k\geq 0}\frac{h^k}{k!}\partial^k_{x},\quad \sigma_{y}=
 \sum_{k\geq 0}\frac{h^k}{k!}\partial^k_{y},\quad \sigma_t=\sum_{k\geq 0}\frac{\tau^k}{k!}\partial^k_t\,.
\]
By the Taylor expansion over a grid point, in the limit when $h$
and $\tau$ go to zero, the right-hand side of (\ref{d_pol})
becomes differential polynomial of the form
\[
 p:=\sum_{\mu}b_\mu \partial^\mu\circ f,\quad b_\mu\in R\,,\quad \partial^\mu\in \{\,\partial_x^i\partial_y^j\partial_t^k\mid i,j,k\in \N_{\geq 0}\,\}.
\]
Thus, $\tilde{p}\,\rhd p\in [F]$. $\Box$
\end{proof}
\begin{corollary}\label{corollary}
Let $\tilde{F}$ be a FDA (\ref{fda}) to (\ref{ins}) and $\succ$ be an admissible order on the monomial set (\ref{monomials}). Suppose
$
(\,\forall \tilde{f}\in \tilde{F}\,)\ \big[\,\tilde{f}\rhd f\in [F]\big]\,.
$
Then, every element $\tilde{p}$ in the difference ideal $[\tilde{F}]$ that admits the representation
\begin{equation}\label{member}
\tilde{q}:=\sum_{k=1}^p\sum_i \tilde{g}_{i,k}\cdot \sigma_i\circ
\tilde{f}_k,\quad \sigma_i\in \Sigma,\quad \tilde{g}_{i,k}\in {\cR}\,,
\end{equation}
where the leading terms of the polynomials in $\sum_i
\tilde{g}_{i,k}\cdot \sigma_i\circ \tilde{f_k}$ do not cancel out,
satisfies $\tilde{q}\,\rhd q\in [F]$.
\end{corollary}

\begin{proof}
Denote by $p_k$ the continuous limit of $\sum_i
\tilde{g}_{i,k}\cdot \sigma_i\circ \tilde{f}_k$. Since $p_k\in
[F]$, the no-cancellation assumption implies $\tilde{p}\,\rhd
\sum_{k=1}^p p_k\in [F]$. $\Box$
\end{proof}

Now we consider three difference approximations to system
(\ref{ins}). The first  two of them were constructed
in~\cite{GB'09} by applying the algorithmic approach to
discretization proposed in~\cite{GBM'06} and based on the finite
volume method combined with numerical integration and difference
elimination. The third approximation is obtained by the
conventional discretization what consists of replacing in
(\ref{ins}) the temporal derivatives with the forward differences
and the spatial derivatives with the central differences.

Every difference equation in an approximation must be satisfied in
every node of the grid. As this takes place, one can apply to
every equation a finite number of the forward shift operators
(\ref{shift}) as well as of their inverses (the backward shift
operators) to transform the approximation into an equivalent form.
Because of this, we consider the  difference approximations
generated in \cite{GB'09} in the form which is commonly used for
numerical solving of PDEs. \vskip 0.1cm \noindent FDA 1
(\cite{GB'09}, Eqs. 13) \vskip 0.1cm \noindent
\[
\left\lbrace
\begin{array}{l}
{e_1}_{j,k}^n := \frac{u_{j+1,k}^n-u_{j-1,k}^n}{2h} + \frac{v_{j,k+1}^n-v_{j,k-1}^n}{2h} = 0, \\[0.2cm]
{e_2}_{j,k}^n := \frac{u_{jk}^{n+1}-u_{jk}^n}{\tau} + \frac{(u_{j+1,k}^n)^2-(u_{j-1,k}^n)^2}{2h}
+ \frac{u_{j,k+1}^n v_{j,k+1}^n -u_{j,k-1}^n v_{j,k-1}^n}{2h} + \frac{p_{j+1,k}^n-p_{j-1,k}^n}{2h} \\[0.2cm]
 \qquad {} - \frac{1}{\mathrm{Re}} \left( \frac{u_{j+2,k}^n-2u_{jk}^n+u_{j-2,k}^n}{4h^2} + \frac{u_{j,k+2}^n-2u_{jk}^n+u_{j,k-2}^n}{4h^2}
\right) = 0,\\[0.2cm] 
{e_3}_{j,k}^n :=\frac{v_{jk}^{n+1}-v_{jk}^n}{\tau} + \frac{u_{j+1,k}^nv_{j+1,k}^n-u_{j-1,k}^n v_{j-1,k}^n}{2h}+
\frac{(v_{j,k+1}^n)^2-(v_{j,k-1}^n)^2}{2h} +  \frac{p_{j,k+1}^n-p_{j,k-1}^n}{2h} \\[0.2cm]
 \qquad  {} - \frac{1}{\mathrm{Re}} \left( \frac{v_{j+2,k}^n-2v_{jk}^n+v_{j-2,k}^n}{4h^2} + \frac{v_{j,k+2}^n-2v_{jk}^n+v_{j,k-2}^n}{4h^2}
\right) = 0,\\[0.2cm]
{e_4}_{j,k}^n := \frac{(u_{j+2,k}^n)^2-2(u_{j,k}^n)^2+(u_{j-2,k}^n)^2}{4h^2} +
\frac{(v_{j,k+2}^n)^2-2(v_{j,k}^n)^2+(v_{j,k-2}^n)^2}{4h^2} \\[0.2cm]
\qquad {} + 2\frac{u_{j+1,k+1}^n v_{j+1,k+1}^n- u_{j+1,k-1}^n v_{j+1,k-1}^n-u_{j-1,k+1}^n v_{j-1,k+1}^n+
u_{j-1,k-1}^n v_{j-1,k-1}^n}{4h^2}\\[0.2cm]
 \qquad {} + \frac{p_{j+2,k}^n-2p_{jk}^n+p_{j-2,k}^n}{4h^2}  + \frac{p_{j,k+2}^n-2p_{jk}^n+p_{j,k-2}^n}{4h^2} = 0\,.
\end{array}
\right.
\]
\vskip 0.1cm
\noindent
FDA 2 (\cite{GB'09}, Eqs. 18)
\vskip 0.1cm
\noindent
\[
\left\lbrace
\begin{array}{l}
{e_1}_{j,k}^n := \frac{u_{j+1,k}^n-u_{j-1,k}^n}{2h} + \frac{v_{j,k+1}^n-v_{j,k-1}^n}{2h} = 0, \\[0.2cm]
{e_2}_{j,k}^n := \frac{u_{jk}^{n+1}-u_{jk}^n}{\tau} + \frac{(u_{j+1,k}^n)^2-({u}_{j-1,k}^n)^2}{2h} + \frac{{u}_{j,k+1}^n{v}_{j,k+1}^n-{u}_{j,k-1}^n{v}_{j,k-1}^n}{2h} + \frac{p_{j+1,k}^n-p_{j-1,k}^n}{2h} \\[0.2cm]
 \qquad {} - \frac{1}{\mathrm{Re}} \left( \frac{u_{j+1,k}^n-2u_{jk}^n+u_{j-1,k}^n}{h^2} + \frac{u_{j,k+1}^n-2u_{jk}^n+u_{j,k-1}^n}{h^2}
\right) = 0,\\[0.2cm] 
{e_3}_{j,k}^n :=\frac{v_{jk}^{n+1}-v_{jk}^n}{\tau} + \frac{{u}_{j+1,k}^n{v}_{j+1,k}^n-{u}_{j-1,k}^n{v}_{j-1,k}^n}{2h} +  \frac{(v_{j,k+1}^n)^2-(v_{j,k-1}^n)^2}{2h} +  \frac{p_{j,k+1}^n-p_{j,k-1}^n}{2h} \\[0.2cm]
 \qquad  {} - \frac{1}{\mathrm{Re}} \left( \frac{v_{j+1,k}^n-2v_{jk}^n+v_{j-1,k}^n}{h^2} + \frac{v_{j,k+1}^n-2v_{jk}^n+v_{j,k-1}^n}{h^2}
\right) = 0,\\[0.2cm]
{e_4}_{j,k}^n := \frac{(u_{j+1,k}^n)^2-2(u_{j,k}^n)^2+(u_{j-1,k}^n)^2}{h^2} +
\frac{(v_{j,k+1}^n)^2-2(v_{j,k}^n)^2+(v_{j,k-1}^n)^2}{h^2} \\[0.2cm]
\qquad {} + 2\frac{u_{j+1,k+1}^n v_{j+1,k+1}^n- u_{j+1,k-1}^n v_{j+1,k-1}^n-u_{j-1,k+1}^n v_{j-1,k+1}^n+
u_{j-1,k-1}^n v_{j-1,k-1}^n}{4h^2}\\[0.2cm]
 \qquad {} + \frac{p_{j+1,k}^n-2p_{jk}^n+p_{j-1,k}^n}{h^2}  + \frac{p_{j,k+1}^n-2p_{jk}^n+p_{j,k-1}^n}{h^2} = 0\,.
\end{array}
\right.
\]
\vskip 0.1cm
\noindent
FDA 3
\vskip 0.1cm
\noindent
\[
\left\lbrace
\begin{array}{l}
{e_1}_{j,k}^n := \frac{u_{j+1,k}^n-u_{j-1,k}^n}{2h} + \frac{v_{j,k+1}^n-v_{j,k-1}^n}{2h} = 0, \\[0.2cm]
{e_2}_{j,k}^n := \frac{u_{jk}^{n+1}-u_{jk}^n}{\tau} + u_{jk}^n \frac{u_{j+1,k}^n-u_{j-1,k}^n}{2h} + v_{jk}^n \frac{u_{j,k+1}^n-u_{j,k-1}^n}{2h} + \frac{p_{j+1,k}^n-p_{j-1,k}^n}{2h} \\[0.2cm]
 \qquad {} - \frac{1}{\mathrm{Re}} \left( \frac{u_{j+1,k}^n-2u_{jk}^n+u_{j-1,k}^n}{h^2} + \frac{u_{j,k+1}^n-2u_{jk}^n+u_{j,k-1}^n}{h^2}
\right) = 0,\\[0.2cm] 
{e_3}_{j,k}^n :=\frac{v_{jk}^{n+1}-v_{jk}^n}{\tau} + u_{jk}^n \frac{v_{j+1,k}^n-v_{j-1,k}^n}{2h} + v_{jk}^n \frac{v_{j,k+1}^n-v_{j,k-1}^n}{2h} +  \frac{p_{j,k+1}^n-p_{j,k-1}^n}{2h} \\[0.2cm]
 \qquad  {} - \frac{1}{\mathrm{Re}} \left( \frac{v_{j+1,k}^n-2v_{jk}^n+v_{j-1,k}^n}{h^2} + \frac{v_{j,k+1}^n-2v_{jk}^n+v_{j,k-1}^n}{h^2}
\right) = 0,\\[0.2cm]
{e_4}_{j,k}^n := \left(\frac{u_{j+1,k}^n-u_{j-1,k}^n}{2h}\right)^2 + 2\frac{v_{j+1,k}^n-v_{j-1,k}^n}{2h} \frac{u_{j,k+1}^n-u_{j,k-1}^n}{2h}
+  \left(\frac{v_{j,k+1}^n-v_{j,k-1}^n}{2h}\right)^2 \\[0.2cm]
 \qquad {} + \frac{p_{j+1,k}^n-2p_{jk}^n+p_{j-1,k}^n}{h^2}  + \frac{p_{j,k+1}^n-2p_{jk}^n+p_{j,k-1}^n}{h^2} = 0
\end{array}
\right.
\]
The difference approximations in the form (\ref{fda}) constructed
in~\cite{GB'09} are obtained from FDA 1,2 by applying the forward
shift operators (\ref{shift}) as follows.
\begin{equation}\label{fda_GB09}
\tilde{F}:=\{\,\sigma\circ {e_1}_{j,k}^n,\sigma^{2}\circ {e_2}_{j,k}^n,\sigma^{2}\circ {e_3}_{j,k}^n,\sigma^{2}
\circ {e_4}_{j,k}^n\,\}\,,\quad \sigma:=\sigma_x\sigma_y\,.
\end{equation}
All three FDAs are $w-$consistent. This can be easily verified by the Taylor expansion of the finite differences in the set
\begin{equation}\label{FDA}
\tilde{F}:=\{{e_1}_{j,k}^n,{e_2}_{j,k}^n,{e_3}_{j,k}^n,{e_4}_{j,k}^n\}
\end{equation}
about the grid point $\{hj,hk,n\tau\}$ when the grid spacings $h$ and $\tau$ go to zero.

To study s-consistency, fix admissible monomial order $\succ$ on
(\ref{monomials}) such that the leading monomials of difference
polynomials in~(\ref{FDA}) read, respectively, as
\begin{equation}\label{LM}
\{\,u_{j+1,k}^n,u_{jk}^{n+1},v_{jk}^{n+1},p_{j+2,k}^n\ \text{for FDA 1 and }p_{j+1,k}^n\ \text{for FDA 2,3}\,\}\,.
\end{equation}
Such monomial order can be easily constructed by extension of the
block (lexdeg) orderly~(cf.~\cite{G'12},\,Remark 2) ranking
$\{\sigma_t\}\{\sigma_x,\sigma_y\}$ with $p\succ u\succ v$.

\begin{proposition}
Among weakly consistent FDAs 1,2, and 3 only FDA 1 is strongly
consistent.
\end{proposition}

\begin{proof}
From the leading monomial set (\ref{LM}) and the structure of FDAs
it follows that every of the approximations has the only
nontrivial $S-$polynomial
\begin{equation}\label{S-poly}
S({e_1}_{j,k}^n, {e_2}_{j,k}^n):= \dfrac{{e_1}_{j,k}^{n+1}}{\tau} - \dfrac{{e_2}_{j+1,k}^n}{2h}\,.
\end{equation}
In the case of FDA 1, the S-polynomial (\ref{S-poly}) is expressed in terms of the difference polynomials in (\ref{FDA}) as follows
\begin{multline} \label{S-poly_1}
S({e_1}_{j,k}^n, {e_2}_{j,k}^n)=\dfrac{{e_1}_{j,k}^{n+1}}{\tau}-\dfrac{{e_2}_{j-1,k}^n}{2h}+\dfrac{{e_3}_{j,k+1}^n -{e_3}_{j,k-1}^n}{2h} \\
{} + \frac{1}{\mathrm{Re}} \left( \frac{{e_1}_{j+2,k}^n-2{e_1}_{jk}^n+{e_1}_{j-2,k}^n}{4h^2} + \frac{{e_1}_{j,k+2}^n-2{e_1}_{jk}^n+{e_1}_{j,k-2}^n}{4h^2}\right)
-{e_4}_{j,k}^n\,.
\end{multline}
The summands in the right-hand side of (\ref{S-poly_1}) have
distinct  leading terms, and thus cannot be cancelled out.
Furthermore, every summand implies a differential consequence of
the corresponding equation in the system (\ref{ins}). Hence, by
Corollary~\ref{corollary}, the $S-$polynomial (\ref{S-poly})
implies, in the continuous limit, an algebraic-differential
consequence of (\ref{ins}).

Consider now an element $\tilde{p}$ of the form (\ref{member}) in the ideal
$[\tilde{F}]\subset {\cR}$ generated by the difference polynomials
appearing in FDA 1. If cancellation occurs in $\tilde{p}$ among
the leading terms, then the sum in (\ref{member}) can be rewritten
by means of the right-hand side of (\ref{S-poly_1}) so that
cancellation of the leading terms cannot occur
(cf.\,\cite{CLO'07}, Ch.2, \S\,6, Th.6). Furthermore, the $S-$polynomial (\ref{S-poly_1}) reduces to zero modulo (\ref{FDA}) if one applies appropriate backshift operators to (\ref{FDA}). Consequently, $\tilde{p}$
implies an element in $[F]$. This  proves $s-$consistency of FDA 1.

\noindent
In the case of FDAs 2 and 3, the corresponding S-polynomial in
addition to an expression similar to the right-hand side of
(\ref{S-poly_1}), i.e. linear in $e_1,e_2,e_3,e_4$, has extra polynomial
additive which we denote by $\Delta_2$ and $\Delta_3$,
respectively. In the continuous limit, $\Delta_2=0$ implies
\[
2 v v_{yyyy}+8 v_y v_{yyy}+6 v_{yy}^2+2 u u_{xxxx}+8 u_x u_{xxx}+6 u_{xx}^2+p_{yyyy}+p_{xxxx}=0\,.
\]
$\Delta_3$ is given by
\[
\Delta_3:=- u_{j,k}^n \dfrac{{e_1}_{j+1,k}^n +{e_1}_{j-1,k}^n}{2h} - v_{j,k}^n \dfrac{{e_1}_{j,k+1}^n +{e_1}_{j,k-1}^n}{2h}+\Delta'_3\,.
\]
Clearly, the explicitly written terms of $\Delta_3$ in the
continuous limit imply an element in the differential ideal $[F]$.
Further, $\Delta'_3=0$ implies PDE
\[
2 v v_{yyyy}+8 v_y v_{yyy}+6 v_{yy}^2+2 u u_{xxxx}+8 u_x u_{xxx}+6 u_{xx}^2+p_{yyyy}+p_{xxxx}=0\,.
\]
The both of obtained differential equations do not follow from the
Navier--Stokes equations that can easily be  verified\footnote{The
Maple library implementing the differential Thomas
decomposition~\cite{DThomas} is capable of making Janet reduction
of a differential polynomial modulo a nonlinear system of PDEs.
The library is free download from {\tt
http://wwwb.math.rwth-aachen.de/thomasdecomposition/index.php}} by
using the Janet reduction of their left-hand sides modulo the
system of polynomials in (\ref{ins}). Therefore, FDAs 2 and 3 are
not strongly consistent. $\Box$
\end{proof}
>From Theorem \ref{th:s-cons}, we immediately conclude:
\begin{corollary}
    A difference standard basis $G$ of the ideal $[\tilde{F}]$ generated by the set (\ref{FDA})
    for FDA 1 satisfies the condition (\ref{cons-gb}).
\end{corollary}

\section{Numerical Comparison}
In this section, we perform some numerical tests for experimental
comparison of the three FDAs of the  previous section. To this
aim, we suppose that the Navier--Stokes system (\ref{ins}) is
defined for $t\ge 0$ in the square domain $\Omega = [0, \, \pi]
\times [0, \, \pi]$ and provide initial conditions for $t=0$ and
boundary conditions for $t>0$ and $(x,y) \in \partial \Omega$.
Initial and boundary conditions are defined according to
(\ref{exact}). Moreover, since we are essentially interested in
the behavior of the different space discretizations used by the
FDAs, any required additional values near the boundary ones are
supposed to be known exactly.

Let $[0, \, \pi] \times [0, \, \pi]$ be discretized in the
$(x,y)$-directions by means of the $(m+2)^2$ equispaced points
$x_j = jh$ and $y_k = kh$, for $j,k=0,\dots m+1$, and $h =
\pi/(m+1)$. Considering difference equations (\ref{FDA}) we
observe that, starting from the initial conditions, the second and
the third equations give explicit formulae to compute
$u_{jk}^{n+1}$ and $v_{jk}^{n+1}$ for $j,k=1,\dots,m$,
respectively. Vice versa, the fourth equation may be used to
derive a $m^2 \times m^2$ linear system that computes the unknowns
$p_{jk}^{n+1}$ for $j,k=1,\dots,m$. In doing so, the first
equation is unnecessary to evaluate the unknowns but may be used
to validate the obtained solution. This procedure may be iterated
for $n=0, 1, \dots, N$ being $t_f = N \tau$ the end point of the
time interval. Since in our experiments we are essentially
interested in comparing different discretizations of $u$, $v$, and
$p$ on the space domain, the value of the time step $\tau$ was
always chosen in order to provide stability.

\noindent
In the following figures, we compare the error behavior in $t_f = 1$
given by the three methods for different values of the Reynolds number
$\mathrm{Re}$. Error is computed by means of the formula\\[-0.1cm]
\[
e_{g} = \max_{j,k} \frac{|g_{j,k}^N-g(x_j,y_k,t_f)|}{1+|g(x_j,y_k,t_f)|}\,.
\]
where $g\in\{u,v,p\}$ and $g(x,y,t)$ belongs to the exact solution~\cite{Chorin} to (\ref{ins})
\begin{eqnarray}
\left\lbrace
\begin{array}{l}
u :=-e^{-2t/\mathrm{Re}} \cos(x) \sin(y)\,,\\
v := e^{-2t/\mathrm{Re}} \sin(x) \cos(y)\,,\\
p :=-e^{-4t/\mathrm{Re}} (\cos(2x)+\cos(2y))/{4}\,.
\end{array}
\right.\label{exact}
\end{eqnarray}

Figure~1 shows the numerical results obtained for (\ref{exact})
with the Reynolds number set to $\mathrm{Re}=10^5$. Each subplot
represents the error of a difference approximation for several
values of $m$ and $N=10$. The three lines in each subplot
represent the error in $u,v$, and $p$. Even if the behavior of the
three schemes is essentially the same, for $m=50$, the scheme
based on FDA 1 is able to obtain the solution with an error less
than $10^{-7}$ while the schemes based on FDAs 2 and 3 do not
obtain an approximation of the solution with an error less than
$10^{-4}$.

\begin{figure} \label{fig1}
\begin{center}
\hspace*{-15mm} \includegraphics[width=1.2\textwidth]{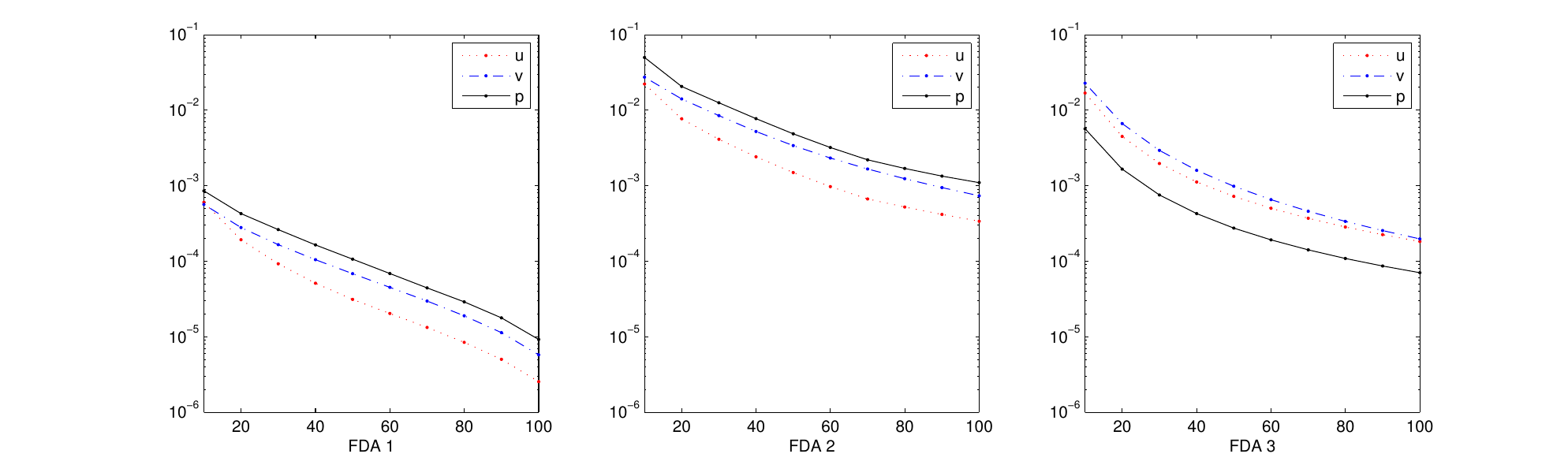}
\end{center}
\caption{Relative error in (\ref{exact}) for FDA 1, FDA 2 and FDA 3 with $N=10$, $t_f=1$, $\mathrm{Re}=10^5$ and varying $m$ from 5 to 50}
\end{figure}

Figure~2 shows the value of the first difference polynomial
${e_1}_{j,k}^n$ in (\ref{FDA}) for the three FDAs and for growing
$m$ obtained by the numerical solution. It is clear that the
discretizations FDA 2 and 3 can not get along without the
continuity equation $f_1$ in the Navier-Stokes system (\ref{ins}).

Figure~3 shows the results obtained for problem (\ref{exact}) with
the Reynolds number set to $\mathrm{Re}=10^2$. Again each subplot
represents the error of a difference scheme and the three lines
inside each subplot represent the error in $u,v$, and $p$,
respectively, for several values of $m$ and $N=40$. Similar
considerations to the previous example may be done: the scheme
based on FDA 1 works much better than the others and the scheme
with FDA 2 is the worst.

\begin{figure} \label{fig2}
\begin{center}
\hspace*{-15mm} \includegraphics[width=.5\textwidth]{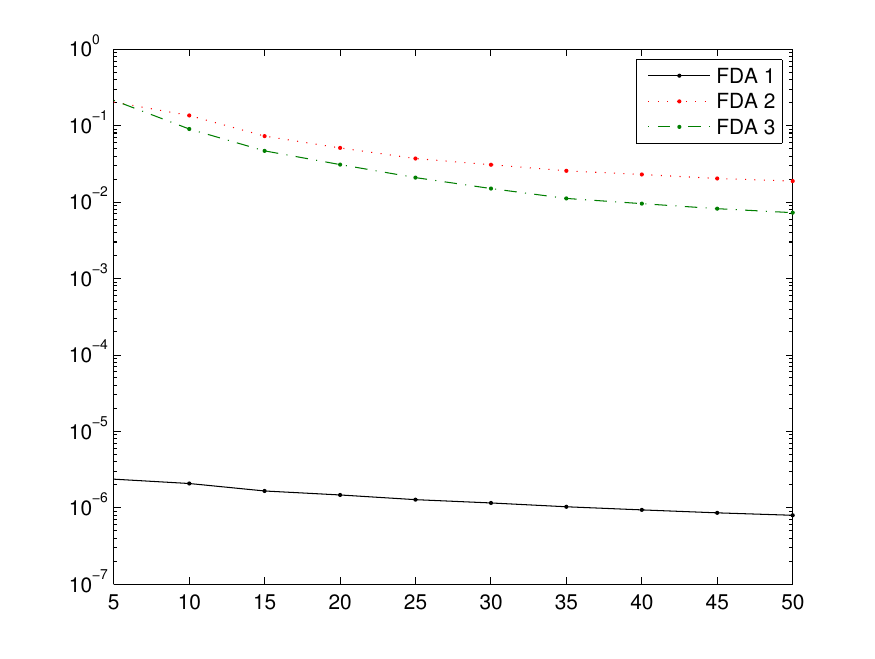}
\end{center}
\caption{Computed value of $f_1$ in (\ref{ins}) for FDA 1 (solid
line),  FDA 2 (dashed line) and FDA 3 (dash-dotted line) with
$N=10$, $t_f=1$, $Re=10^5$ and varying $m$ from 5 to 50}
\end{figure}

\begin{figure} \label{fig3}
\begin{center}
\hspace*{-15mm}
\includegraphics[width=1.2\textwidth]{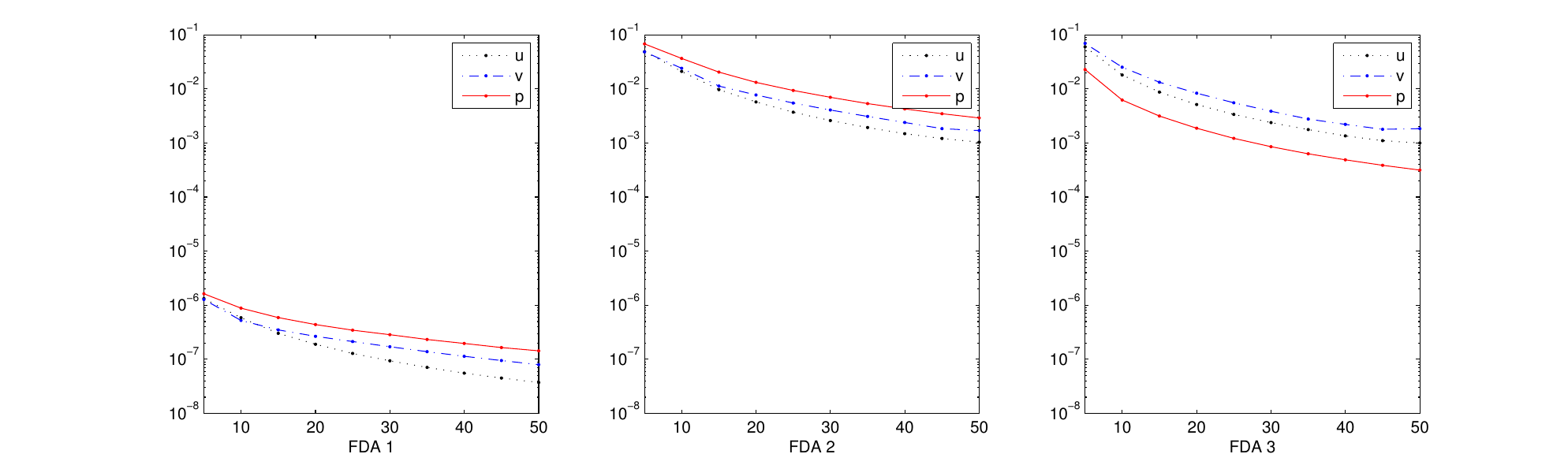}
\end{center}
\caption{Computed errors in $u,v$ and $p$ for FDA 1 (left), 2
(middle) and 3 (right): $N=40$, $t_f=1$, $Re=10^2$ and varying $m$
from 10 to 100}
\end{figure}

\noindent We conclude showing in Figure~4 the computed error in
$t_f=1$ using the $s-$consistent FDA 1 applied to the
problem~(\ref{exact}) ($\mathrm{Re}=10^2$) with $N=40$ and
$m=100$. Larger errors are near the boundaries, and $u$ and $v$
seem to be better approximated than $p$.

\begin{figure} \label{fig4}
\begin{center}
\hspace*{-15mm}
\includegraphics[width=1.2\textwidth]{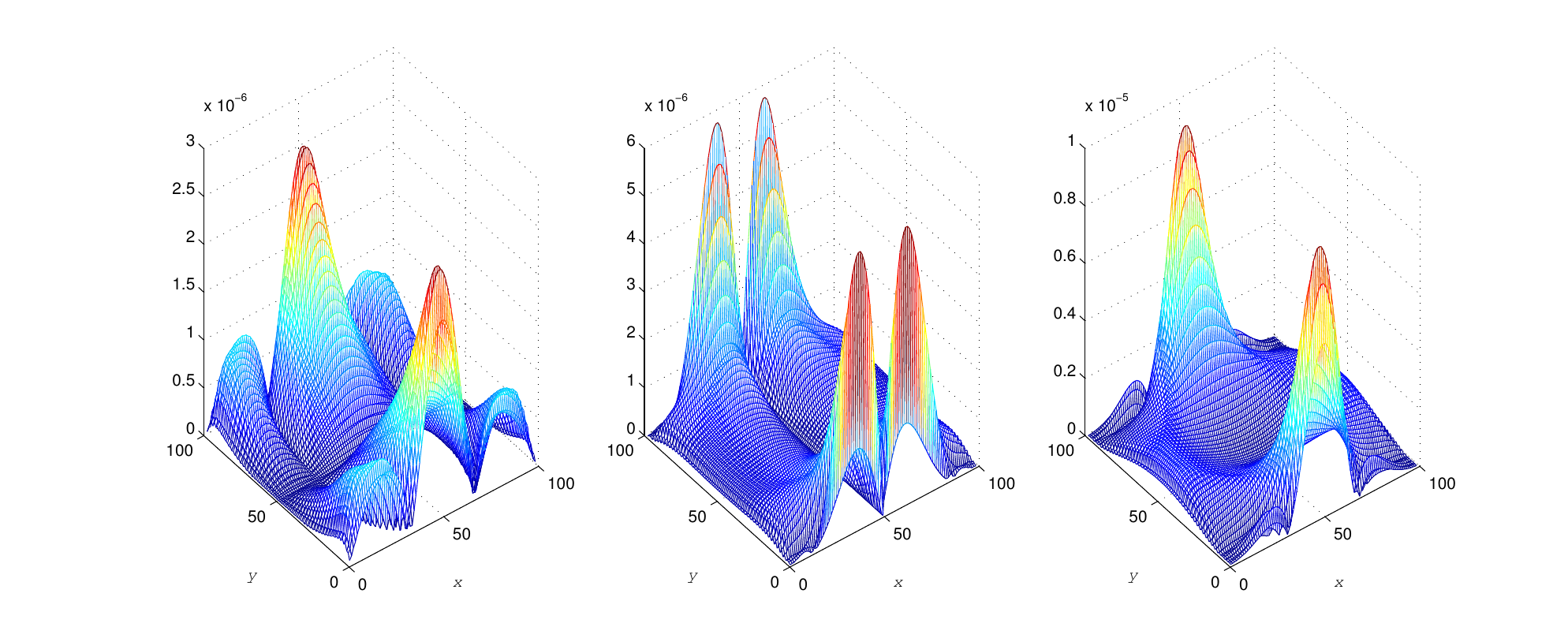}
\end{center}
\caption{Computed error with FDA 1 ($u$, $v$ and $p$,
respectively): $N=40$, $t_f=1$, $Re=10^2$ and $m=100$}
\end{figure}

\section{Conclusion}
As it has been already demonstrated in~\cite{GR'10} for
overdetermined systems of linear PDEs, it may be highly nontrivial
to construct strongly consistent difference approximations. In the
given paper, we have demonstrated that the demands of
s-consistency impose strong limitations on the finite difference
approximations to the nonlinear system of Navier--Stokes
equations. These limitations proceed from the fact that
$s-$consistent approximations inherit at the discrete level all
basic algebraic properties of the initial differential equations.

It turned out that among two distinctive approximations generated
in~\cite{GBM'06} (by applying the same algorithmic technique with
different choice of numerical integration method), the one with a
$5\times 5$ stencil (FDA 1) is strongly consistent whereas the
other one with a $3\times 3$ stencil (FDA 2) is not. This result
is at variance with universally accepted opinion that
discretization with a more compact stencil is numerically
favoured. One more discretization with a $3\times 3$ stencil (FDA
3), obtained from the differential Navier--Stokes equations by the
replacement of spatial derivatives with the central differences
and of temporal derivatives with the forward differences, also
failed to be $s-$consistent. As this takes place, our computer
experimentation revealed much better numerical behavior of the
$s-$consistent approximation in comparison with the considered
s-inconsistent ones. The question of existence of $s-$consistent FDA
to (\ref{ins}) with a $3\times 3$ stencil is open.

Unlike the linear case~\cite{GR'10}, given a difference
approximation on a grid with equisized grid spacings, one cannot
fully algorithmically check its s-consistency. This is owing to
non-noetherianity of difference polynomial rings that may lead to
non-existence of a finite difference \Gr basis for the ideal
generated by the approximation. And even with the existence of a
\Gr basis, its construction and algorithmic verification of
s-consistency may be very hard. For example, by using experimental
implementation in Maple~\cite{L'12} of the algorithm of
papers~\cite{L'12,GLS'13}\footnote{To our knowledge, it is the
only software available for computation of difference \Gr bases
for the nonlinear case.}, many finite \Gr bases have been
constructed for the $s-$consistent approximation FDA-1 and for many
different monomial orders. In doing so, the smallest obtained
basis consists of 5 different polynomials, and one of the
polynomials has 404 terms. In distinction to those rather tedious
computations, the verification of s-consistency for FDA 1 and
s-inconsistency for the other two was done by analysing  the only
S-polynomial and required much less symbolic computation.

It should be noted that in our paper, we use the collocated
arrangement of the dependent variables $u,v$, and $p$ in the
system (\ref{ins}) that often gives rise to oscillations of the
variables (cf.~\cite{FP'96}) and  makes  impossible convergence of
numerical solutions. Our experiments presented in Section 5
demonstrate no spurious oscillations of the numerical solution.
This can be considered as a significant positive property of the
obtained FDAs.

\section{Acknowledgements}
The authors thank the anonymous referees for constructive comments
and recommendations which helped to improve the readability and
quality of the paper. The contribution of the second and third
authors (Yu.B. and V.G.) was partially supported by grant
13-01-00668 from the Russian Foundation for Basic Research and by
grant 3802.2012.2 from the Ministry of Education and Science of
the Russian Federation.

\end{document}